\documentclass[12pt]{article}

\usepackage[T1]{fontenc}
\usepackage{lmodern}
\usepackage{microtype}
\usepackage{geometry}
\usepackage{amsmath,amssymb,amsthm}

\allowdisplaybreaks

\newtheorem{theorem}{Theorem}
\newtheorem{lemma}[theorem]{Lemma}

\newcommand{\Z}{\mathbb Z}
\newcommand{\Pres}{\mathsf{Pres}}
\newcommand{\Lim}{\mathsf{Lim}}
\newcommand{\Colim}{\mathsf{Colim}}
\newcommand{\Blim}{\mathsf{Blim}}
\newcommand{\im}{\operatorname{Im}}
\newcommand{\fideal}{\mathbf f}
\newcommand{\rideal}{\mathbf r}
\newcommand{\gideal}{\mathbf g}

\title{Dimension quotients as boundary limits: the general case.}
\author{Roman Mikhailov}
\date{}

\begin{document}
\maketitle

\begin{abstract}
We identify, functorially, the boundary limit of a simply defined presentation functor with the dimension quotient, both for groups and for Lie rings over integers. 
\end{abstract}

\section{Introduction}

Let $G$ be a group, $\Z[G]$ its integral group ring, and 
$\gideal$ the augmentation ideal.  The dimension subgroups of $G$ are
\[
 D_n(G)=G\cap(1+\gideal^n).
\]
The lower central series of $G$ is defined by
\[
 \gamma_1(G)=G,\qquad \gamma_2(G)=G'=[G,G],\qquad
 \gamma_{k+1}(G)=[\gamma_k(G),G]\quad(k\geq1).
\]
The dimension subgroups originate in the work of Magnus on free groups and noncommutative
power series \cite{Magnus1937}.  The inclusion
\[
 \gamma_n(G)\subseteq D_n(G)
\]
is immediate from the commutator identity in the integral group ring.
The reverse inclusion holds for $n\leq3$, and the assertion that it
holds in every degree was known as the dimension subgroup conjecture.
Rips disproved the conjecture by constructing a group with
$D_4(G)\neq\gamma_4(G)$ \cite{Rips1972}.  Nontrivial dimension
quotients are now known in every degree $n\geq4$; see
\cite{MikhailovPassi2009} for the history and the classical methods.

The dimension quotients
\[
 G\longmapsto D_n(G)/\gamma_n(G)
\]
are complicated functors.  Their integral torsion reflects derived
functors of nonadditive functors, group homology, and homotopical
phenomena.  Homological and simplicial approaches were developed in
\cite{HartlMikhailovPassi2008}; connections with homotopy groups of
spheres are studied in \cite{BartholdiMikhailov2023}; see also
\cite{MikhailovICM}.

We recall two elementary categorical constructions used below. Let
$Q\colon\mathcal C\to\mathcal A$ be a functor between categories $\mathcal C$ and $\mathcal A$.  A limit of $Q$ is an
object $\Lim_{\mathcal C}\ Q$ together with morphisms
\[
 p_c\colon\Lim_{\mathcal C}\ Q\longrightarrow Q(c),
 \qquad Q(\alpha)\circ p_c=p_d
 \quad(\alpha\colon c\to d).
\]
It is characterized by the following universal property: if
$T\in\mathcal A$ and morphisms $f_c\colon T\to Q(c)$ satisfy
$Q(\alpha)\circ f_c=f_d$ for every $\alpha\colon c\to d$, then there is
a unique morphism
\[
 f\colon T\longrightarrow\Lim_{\mathcal C}\ Q
 \qquad\text{such that}\qquad p_c\circ f=f_c
\]
for every $c$.

Dually, a colimit of $Q$ is an object $\Colim_{\mathcal C}\ Q$
together with morphisms
\[
 j_c\colon Q(c)\longrightarrow\Colim_{\mathcal C}\ Q,
 \qquad j_d\circ Q(\alpha)=j_c
 \quad(\alpha\colon c\to d).
\]
Its universal property says that if morphisms
$g_c\colon Q(c)\to T$ satisfy $g_d\circ Q(\alpha)=g_c$, then there is
a unique morphism
\[
 g\colon\Colim_{\mathcal C}\ Q\longrightarrow T
 \qquad\text{such that}\qquad g\circ j_c=g_c
\]
for every $c$.  These universal properties determine the limit and the
colimit uniquely up to unique isomorphism.

When $\mathcal A$ is the category of groups or of Lie rings, these
objects have concrete descriptions.  The limit is the group or Lie
ring of compatible families
\[
 \Lim_{\mathcal C}\ Q
 =
 \left\{
 (q_c)_{c\in\mathcal C}\in\prod_{c\in\mathcal C}Q(c)
 \ \middle|\
 Q(\alpha)(q_c)=q_d
 \text{ for every }\alpha\colon c\to d
 \right\}.
\]
The colimit is obtained by putting all values of $Q$ together and
imposing the relations prescribed by the morphisms:
\[
 \Colim_{\mathcal C}\ Q
 =
 \left(\coprod_{c\in\mathcal C}Q(c)\right)\Big/\!
 \left\langle
 \iota_c(q)\sim\iota_d\bigl(Q(\alpha)(q)\bigr)
 \ \middle|\
 \alpha\colon c\to d,\ q\in Q(c)
 \right\rangle.
\]
Here the coproduct and the generated relation are taken in
$\mathcal A$: for groups one takes the free product and the normal
closure of the relations, while for Lie rings one takes the coproduct
and the ideal generated by their differences.  Thus the limit retains
families compatible with every morphism, whereas the colimit identifies
all elements that the morphisms force to agree.

For a general category $\mathcal C$, there need not be a
canonical morphism from the limit to the colimit.  Suppose, however,
that $\mathcal C$ is nonempty and connected, meaning that any two
objects can be joined by a zigzag of morphisms.  For every object $c$,
consider
\[
 \theta_c\colon
 \Lim_{\mathcal C}\ Q\xrightarrow{\ p_c\ }Q(c)
 \xrightarrow{\ j_c\ }\Colim_{\mathcal C}\ Q.
\]
If $\alpha\colon c\to d$, then
\[
 j_d\circ p_d
 =j_d\circ Q(\alpha)\circ p_c
 =j_c\circ p_c.
\]
Connectedness therefore implies that $\theta_c$ is independent of
$c$; denote the resulting canonical morphism by
\[
 \theta_Q\colon\Lim_{\mathcal C}\ Q\longrightarrow
 \Colim_{\mathcal C}\ Q.
\]
If $\mathcal A$ is the category of groups or of Lie rings, the image
of this morphism exists.  Following \cite{Mikhailov2024}, define the
boundary limit by
\[
 \Blim_{\mathcal C}\ Q
 =\im\bigl(\theta_Q\colon
 \Lim_{\mathcal C}\ Q\longrightarrow\Colim_{\mathcal C}\ Q\bigr).
\]

A complementary approach is to study functors on categories of free
presentations.  Write $\Pres(G)$ for the category whose objects are
exact sequences
\[
 1\longrightarrow R\longrightarrow F\longrightarrow G\longrightarrow1,
 \qquad F\text{ free}.
\]
The morphisms in $\Pres(G)$ induce the identity on $G$.
A functor built from
$F$ and $R$ usually depends on the chosen presentation.  The limit and
the colimit compare its values over all free presentations. 
Between any two objects there is a morphism, obtained by lifting a free
basis of the source to the target.  Thus $\Pres(G)$ is connected, and
the boundary limit is defined for every group-valued functor on it.
Limits and higher limits over $\Pres(G)$
recover group homology and derived functors of nonadditive functors;
the systematic theory and the associated $\mathbf{fr}$-language are
developed in
\cite{IvanovMikhailov2018}.
Colimits and higher colimits over categories of free presentations were
studied in
\cite{IvanovMikhailovSosnilo2019}.

The boundary-limit construction uses both ends of this presentation
calculus.
For every $n$, the following natural monomorphism was obtained in
\cite{Mikhailov2024}:
\[
 \Blim_{\Pres(G)}\ \frac{F}{R'\gamma_n(F)}
 \lhook\joinrel\longrightarrow\frac{D_n(G)}{\gamma_n(G)},
\]
It was proved there that this map is an isomorphism for $n=4$, and the
equality was proposed in all degrees.  The purpose of this paper is to
prove that equality.

\vspace{.3cm}
\noindent {\bf Main Theorem.} {\it 
For every group $G$ and every $n\geq1$,
\[
 \Blim_{\Pres(G)}\ \frac{F}{R'\gamma_n(F)}
 \cong\frac{D_n(G)}{\gamma_n(G)}
\]
naturally in $G$.}

\vspace{.3cm}

The theorem also gives an intrinsic description of the dimension
quotient which does not use the integral group ring.  Fix one
presentation $c=(F,R)$, put
\[
 Q_n(c)=F/(R'\gamma_n(F)),
\]
and let $i_0,i_1:c\to c\sqcup c$ be the two coproduct inclusions.
Combining the theorem with the equalizer formula for limits gives
\[
 \frac{D_n(G)}{\gamma_n(G)}
 \cong
 \im\left(
  \operatorname{Eq}\bigl(
   Q_n(c)\mathrel{\substack{\xrightarrow{\ Q_n(i_0)\ }\\[-0.4ex]
                              \xrightarrow[\ Q_n(i_1)\ ]{}}}
   Q_n(c\sqcup c)
  \bigr)
  \longrightarrow G/\gamma_n(G)
 \right).
\]
Thus, after $c$ is fixed, the dimension quotient is computed inside
$Q_n(c)$ and $Q_n(c\sqcup c)$, without group rings.  The Lie-ring theorem below
gives the analogous presentation-theoretic description, with
$F/([R,R]+\gamma_n(F))$ as the presentation functor, without using a
universal enveloping algebra.

The proof identifies the entire limit inside a fixed presentation.  Its
essential point is a Fox-calculus formula in the relation module of the
coproduct of two presentations.  The same argument has an additive
form for Lie rings over $\Z$.  Section~2 proves the group theorem, and
Section~3 records the Lie-ring version and the simplifications in its
proof.

\section{Proof of the main Theorem}

The entire scheme of the proof is already present in
\cite{Mikhailov2024}.  For the general statement, however, that paper
does not supply an all-degree proof of the key step: one has to show
that the functor
\[
 (F,R)\longmapsto
 \frac{F\cap(1+\rideal\fideal+\fideal^n)}
      {R'\gamma_n(F)}
\]
is constant on $\Pres(G)$, up to its canonical identifications.  In
\cite{Mikhailov2024} this constancy is established only for $n=4$ by a
degree-specific argument, and the passage to the boundary limit is
formulated there through a higher limit of a non-abelian presentation
functor.  Here we prove the
constancy for every $n$ and give direct proofs of all the remaining
steps.  In particular, we use neither higher limits of group-valued
functors nor any results from the theory of derived limits; every fact
about ordinary limits and colimits needed below is proved directly from
their definitions.

The assertion for $n=1$ is immediate: both sides are trivial.  Fix
$n\geq2$ and retain the notation $Q_n(F,R)=F/(R'\gamma_n(F))$.

\begin{lemma}\label{lem:formal}
Fix a presentation $c=(F,R)$.  If $c\sqcup c$ is its coproduct with itself in
$\Pres(G)$ and $i_0,i_1:c\to c\sqcup c$ are the inclusions, then, for
every group-valued functor $T$,
\[
 \Lim\ T=\{a\in T(c):T(i_0)(a)=T(i_1)(a)\}.
\]
Moreover,
\[
 \Colim\ Q_n\cong G/\gamma_n(G).
\]
\end{lemma}

\begin{proof}
For any two presentations $d=(F_d,R_d)$ and $e=(F_e,R_e)$, a morphism
$d\to e$ is obtained by choosing, for every element of a free basis of
$F_d$, a lift in $F_e$ of its image in $G$.  In particular, there is a
morphism from $c$ to every object of $\Pres(G)$.

Let $a\in T(c)$ satisfy $T(i_0)(a)=T(i_1)(a)$.  If
$\alpha,\beta:c\to d$, the universal property of the coproduct gives a
morphism $h:c\sqcup c\to d$ with $h i_0=\alpha$ and $h i_1=\beta$;
hence
\[
 T(\alpha)(a)=T(h)T(i_0)(a)=T(h)T(i_1)(a)=T(\beta)(a).
\]
Choose a morphism $\alpha_d:c\to d$ for each $d$ and put
$a_d=T(\alpha_d)(a)$.  If $\lambda:d\to e$, the preceding equality,
applied to the two maps $\lambda\alpha_d,\alpha_e:c\to e$, gives
$T(\lambda)(a_d)=a_e$; thus $(a_d)_d$ is a compatible family.  Conversely,
if $(a_d)_d$ is compatible, then
$T(i_0)(a_c)=a_{c\sqcup c}=T(i_1)(a_c)$.  These two constructions are
inverse and prove the asserted description of the limit.

For a presentation $d=(F_d,R_d)$, the map
$F_d\to G/\gamma_n(G)$ kills $R_d\gamma_n(F_d)$ and hence induces
\[
 u_d:Q_n(F_d,R_d)\longrightarrow G/\gamma_n(G).
\]
The maps $u_d$ commute with all morphisms of presentations, so the
universal property of the colimit gives a homomorphism
\[
 u:\Colim\ Q_n\longrightarrow G/\gamma_n(G).
\]

We construct its inverse.  Fix $d=(F_d,R_d)$ and $r\in R_d$, and form
the presentation
\[
 \widetilde F=F_d*\langle t\rangle\longrightarrow G,
 \qquad t\longmapsto1.
\]
There are two morphisms from this presentation to $d$ which restrict
to the identity on $F_d$ and send $t$ respectively to $r$ and to $1$.
The defining relations of the colimit therefore identify the class of
$r$ with the identity; thus every element of every relation subgroup
$R_d$ has trivial colimit class.

For $g\in G$, choose a presentation $d=(F_d,R_d)$ and a lift
$f\in F_d$ of $g$, and define $v(g)$ to be the colimit class of the
image of $f$ in $Q_n(F_d,R_d)$.  If $f'$ is another lift in $F_d$, then
$f'f^{-1}\in R_d$, so $f$ and $f'$ have the same colimit class.  If
$e=(F_e,R_e)$ is another presentation, choose a morphism
$\alpha:d\to e$; the colimit identifies the class of $f$ with that of
$\alpha(f)$, and $\alpha(f)$ differs from any chosen lift of $g$ in
$F_e$ by an element of $R_e$.  Hence $v(g)$ is independent of every
choice.  Choosing lifts of two elements of $G$ in one presentation
shows that $v:G\to\Colim\ Q_n$ is a homomorphism.

Every epimorphism $F_d\to G$ maps $\gamma_n(F_d)$ onto
$\gamma_n(G)$; this follows by induction on $n$ from the definition of
the lower central series.  Since $\gamma_n(F_d)$ is trivial in
$Q_n(F_d,R_d)$, the homomorphism $v$ kills $\gamma_n(G)$ and induces
\[
 \bar v:G/\gamma_n(G)\longrightarrow\Colim\ Q_n.
\]
The composite $u\bar v$ is the identity by the definition of $v$.
The colimit is generated by the images of the groups $Q_n(F_d,R_d)$,
and on the class of every $f\in F_d$ the composite $\bar v u$ is also
the identity.  Thus $u$ and $\bar v$ are inverse isomorphisms.
\end{proof}

For a presentation $(F,R)$, let
\[
 \fideal=\ker(\Z[F]\to\Z),\qquad
 \rideal=\ker(\Z[F]\to\Z[G])=(R-1)\Z[F].
\]
The ideal $\rideal\fideal+\fideal^n$ is two-sided; hence
$F\cap(1+\rideal\fideal+\fideal^n)$ is normal in $F$.

\begin{lemma}\label{lem:correction}
One has
\[
 R'\gamma_n(F)\subseteq
 F\cap(1+\rideal\fideal+\fideal^n),
 \qquad
 \im\bigl(F\cap(1+\rideal\fideal+\fideal^n)\to G\bigr)=D_n(G).
\]
\end{lemma}

\begin{proof}
Use the convention $[u,v]=u^{-1}v^{-1}uv$.  In $\Z[F]$,
\[
 [u,v]-1=u^{-1}v^{-1}
 \bigl((u-1)(v-1)-(v-1)(u-1)\bigr).
\]
For $u,v\in R$ the right-hand side lies in
$\rideal\fideal$, so
$R'\subseteq F\cap(1+\rideal\fideal)$.  If
$u-1\in\fideal^k$ and $v\in F$, the same identity gives
$[u,v]-1\in\fideal^{k+1}$; induction therefore yields
$\gamma_n(F)\subseteq F\cap(1+\fideal^n)$.  Finally, for
$a\in R'$ and $b\in\gamma_n(F)$,
\[
 ab-1=(a-1)+(b-1)+(a-1)(b-1)
       \in\rideal\fideal+\fideal^n.
\]
This proves the first inclusion.

We will use
\begin{equation}\label{eq:r-reduction}
 \rideal=(R-1)+\rideal\fideal
\end{equation}
as an equality of subsets of $\Z[F]$.  Indeed, modulo
$\rideal\fideal$,
\[
 \sum_j(r_j-1)a_j
 \equiv\sum_j\varepsilon(a_j)(r_j-1).
\]
The map $R\to\rideal/\rideal\fideal$, $r\mapsto r-1$, is a group
homomorphism from the multiplicative group to the additive group,
because
\[
 rs-1\equiv(r-1)+(s-1)\pmod{\rideal\fideal}.
\]
Hence the last sum is congruent to $\rho-1$ for some $\rho\in R$,
which proves \eqref{eq:r-reduction}.

If $v\in F\cap(1+\rideal\fideal+\fideal^n)$, its image $x\in G$ satisfies
$x-1\in\gideal^n$, so $x\in D_n(G)$.  Conversely, let
$x\in D_n(G)$ and lift it to $w\in F$.  The inverse image of
$\gideal^n$ in $\Z[F]$ is $\rideal+\fideal^n$, because the quotient
map sends $\fideal^n$ onto $\gideal^n$.  Thus
$w-1\in\rideal+\fideal^n$.  By \eqref{eq:r-reduction}, choose
$\rho\in R$ such that
\[
 w-\rho\in\rideal\fideal+\fideal^n.
\]
Then $v=w\rho^{-1}$ maps to $x$ and
\[
 v-1=(w-\rho)\rho^{-1}\in\rideal\fideal+\fideal^n.
\]
Thus $v\in F\cap(1+\rideal\fideal+\fideal^n)$.
\end{proof}

Choose a free basis $X$ of $F$.  The left Fox partial derivatives
$\frac{\partial}{\partial x}:\Z[F]\to\Z[F]$ are determined by
\[
 \frac{\partial y}{\partial x}=\delta_{x,y},\qquad
 \frac{\partial(ab)}{\partial x}
 =\frac{\partial a}{\partial x}\varepsilon(b)
  +a\frac{\partial b}{\partial x}.
\]
These are the usual Fox partial derivatives \cite{Fox1953}.

\begin{lemma}\label{lem:fox}
For $a\in\Z[F]$,
\begin{equation}\label{eq:fundamental}
 a-\varepsilon(a)=\sum_{x\in X}
 \frac{\partial a}{\partial x}(x-1).
\end{equation}
Furthermore,
\begin{equation}\label{eq:derivative-filtration}
 \frac{\partial}{\partial x}
 (\rideal\fideal+\fideal^n)
 \subseteq\rideal+\fideal^{n-1}.
\end{equation}
Finally,
\[
 \Lim\ \frac{\fideal}{\rideal\fideal+\fideal^n}=0.
\]
\end{lemma}

\begin{proof}
The two sides of \eqref{eq:fundamental}, as functions of $a$, agree at
$1$ and on the free generators.  Both are additive, and the value of each at
$ab$ is its value at $a$ multiplied by $\varepsilon(b)$ plus $a$
times its value at $b$.  This rule also determines the value on every
inverse by applying it to $uu^{-1}=1$.  The two sides therefore agree
on every word and, by linearity, on $\Z[F]$.

If $a\in\rideal$ and $b\in\fideal$, then
\[
 \frac{\partial(ab)}{\partial x}
 =a\frac{\partial b}{\partial x}\in\rideal.
\]
Thus $\frac{\partial}{\partial x}(\rideal\fideal)\subseteq\rideal$.
Repeated application of
the product rule gives
$\frac{\partial}{\partial x}(\fideal^n)
\subseteq\fideal^{n-1}$, proving
\eqref{eq:derivative-filtration}.

It remains to prove the vanishing.  Let
\[
 H=F_0*F_1\longrightarrow G,\qquad S=\ker(H\to G),
\]
be the coproduct presentation, and denote its augmentation and relation
ideals by $\mathbf h$ and $\mathbf s$.  If the class of
$a\in\fideal$ equalizes the two coproduct maps, then
\[
 i_0(a)-i_1(a)\in\mathbf s\mathbf h+\mathbf h^n.
\]
Take the Fox derivative with respect to $i_0(x)$.  The derivative of
$i_1(a)$ is zero, and the defining rules for Fox derivatives give
$\partial i_0(a)/\partial i_0(x)=i_0(\partial a/\partial x)$; hence
\eqref{eq:derivative-filtration} in $H$ gives
\[
 i_0\!\left(\frac{\partial a}{\partial x}\right)
 \in\mathbf s+\mathbf h^{n-1}.
\]
After mapping to $\Z[G]$, the image of
$\frac{\partial a}{\partial x}$ lies in
$\gideal^{n-1}$.  Since the inverse image of $\gideal^{n-1}$ in
$\Z[F]$ is $\rideal+\fideal^{n-1}$,
\[
 \frac{\partial a}{\partial x}
 \in\rideal+\fideal^{n-1}
\]
for every $x$.  Formula \eqref{eq:fundamental}, together with
$\varepsilon(a)=0$, now gives
$a\in\rideal\fideal+\fideal^n$.  The equalizer is zero, and
Lemma~\ref{lem:formal} gives the asserted vanishing.
\end{proof}

\begin{lemma}\label{lem:representatives}
Every element of $\Lim\ Q_n$, evaluated at $(F,R)$, is represented by
an element of $F\cap(1+\rideal\fideal+\fideal^n)$.
\end{lemma}

\begin{proof}
There is a natural injection of pointed sets
\[
 \frac{F}{F\cap(1+\rideal\fideal+\fideal^n)}
 \longrightarrow
 \frac{\fideal}{\rideal\fideal+\fideal^n},
 \qquad
 u\bigl(F\cap(1+\rideal\fideal+\fideal^n)\bigr)
 \longmapsto u-1+(\rideal\fideal+\fideal^n).
\]
Indeed,
\[
 (u-1)-(v-1)\in\rideal\fideal+\fideal^n
 \quad\Longleftrightarrow\quad
 uv^{-1}-1=(u-v)v^{-1}
 \in\rideal\fideal+\fideal^n.
\]
The displayed maps form an objectwise injection of pointed-set-valued
functors.  Hence they induce an injection on limits: explicitly, a
compatible family whose image is the zero family must be the base-point
family because every displayed map is injective.  Lemma~\ref{lem:fox}
therefore implies
\[
 \Lim\ \frac{F}{F\cap(1+\rideal\fideal+\fideal^n)}=*.
\]
The map
\[
 Q_n\longrightarrow
 \frac{F}{F\cap(1+\rideal\fideal+\fideal^n)}
\]
is defined by Lemma~\ref{lem:correction}.  A limit element maps to the
base point, and hence has a representative in
$F\cap(1+\rideal\fideal+\fideal^n)$.
\end{proof}

\begin{lemma}\label{lem:fox-coefficients}
Let $F=F(X)$, let $\varepsilon:\Z[F]\to\Z$ be the augmentation, and
let $\varphi:F\to G$ be a homomorphism.  For every family
$(h_x)_{x\in X}$ of elements of $\Z[F]$, there is a unique additive map
$D:\Z[F]\to\Z[F]$ such that
\[
 D(ab)=D(a)\varepsilon(b)+aD(b),
 \qquad D(x)=h_x
\]
for all $a,b\in\Z[F]$ and $x\in X$.  It is given by
\begin{equation}\label{eq:fox-theorem}
 D(a)=\sum_{x\in X}
 \frac{\partial a}{\partial x}\,h_x,
 \qquad a\in\Z[F].
\end{equation}
For each $a$ the sum is finite.

Consequently, if $A$ is a left $\Z[G]$-module and
$d:F\to A$ satisfies
\[
 d(uv)=d(u)+\varphi(u)d(v),
 \qquad u,v\in F,
\]
then
\begin{equation}\label{eq:fox-with-coefficients}
 d(u)=\sum_{x\in X}
 \varphi_*\!\left(\frac{\partial u}{\partial x}\right)d(x),
 \qquad u\in F,
\end{equation}
where $\varphi_*:\Z[F]\to\Z[G]$ is induced by $\varphi$.
\end{lemma}

\begin{proof}
The first assertion, including existence, uniqueness, and
formula~\eqref{eq:fox-theorem}, is precisely Fox's theorem
\cite[\S2, Theorem, (2.2)]{Fox1953}.

Let $d_0(u)$ denote the right-hand side of
\eqref{eq:fox-with-coefficients}.  It is a finite sum because a reduced
word for $u$ involves only finitely many generators.  For $u,v\in F$,
the Fox product rule gives
\[
 \frac{\partial(uv)}{\partial x}
 =\frac{\partial u}{\partial x}
  +u\frac{\partial v}{\partial x}.
\]
Therefore
\[
 \begin{aligned}
 d_0(uv)
 &=\sum_{x\in X}
   \varphi_*\!\left(\frac{\partial u}{\partial x}
   +u\frac{\partial v}{\partial x}\right)d(x)\\
 &=d_0(u)+\varphi(u)d_0(v).
 \end{aligned}
\]
Thus $d_0$ is a $\varphi$-derivation, and
$d_0(x)=d(x)$ for every $x\in X$.

A $\varphi$-derivation is determined by its values on $X$: it satisfies
$d(1)=0$ and
\[
 d(x^{-1})=-\varphi(x)^{-1}d(x),
\]
the latter identity following from $0=d(xx^{-1})$.  Hence two
$\varphi$-derivations which agree on $X$ agree on $X^{\pm1}$ and then,
by induction on word length, on all of $F$.  Consequently $d_0=d$, which
proves \eqref{eq:fox-with-coefficients}.
\end{proof}

\medskip

\begin{lemma}\label{lem:stability}
Every $v\in F\cap(1+\rideal\fideal+\fideal^n)$ represents an element of
$\Lim\ Q_n$.
\end{lemma}

\begin{proof}
Use the coproduct presentation
\[
 H=F_0*F_1\longrightarrow G,\qquad S=\ker(H\to G),
\]
with inclusions $i_0,i_1:F\to H$.  By Lemma~\ref{lem:formal}, we only
have to prove that the two images of $v$ agree in
$Q_n(H,S)=H/(S'\gamma_n(H))$, or equivalently that
\[
 i_1(v)i_0(v)^{-1}\in S'\gamma_n(H).
\]

The relation module $M=S/S'$ is the linearization of the kernel $S$.
We write it additively and give it its usual left $\Z[G]$-module
structure by conjugation.  Namely, if $h\in H$ lifts $g\in G$, then
$g\cdot(sS')=hsh^{-1}S'$.  This is well defined because changing $h$
by an element of $S$ changes the conjugate only by an element of $S'$.

For $u\in F$, define its difference in the two copies of $F$ by
\[
 \Delta(u)=i_1(u)i_0(u)^{-1}S'\in M.
\]
It belongs to $M$ because $i_0(u)$ and $i_1(u)$ have the same image in
$G$.  The proof is the following chain of implications:
\[
 v-1\in\rideal\fideal+\fideal^n
 \ \Longrightarrow\ 
 \overline{\frac{\partial v}{\partial x}}\in\gideal^{n-1}
 \ \Longrightarrow\ 
 \Delta(v)\in\gideal^{n-1}M
 \ \Longrightarrow\ 
 i_1(v)i_0(v)^{-1}\in S'\gamma_n(H).
\]
We prove the three implications separately.

\emph{Step 1.}  Since
$v-1\in\rideal\fideal+\fideal^n$, Lemma~\ref{lem:fox},
\eqref{eq:derivative-filtration}, gives
\[
 \frac{\partial v}{\partial x}
 =\frac{\partial(v-1)}{\partial x}
 \in\rideal+\fideal^{n-1}.
\]
After projecting to $\Z[G]$, the $\rideal$-part vanishes and hence
\[
 \overline{\frac{\partial v}{\partial x}}\in\gideal^{n-1}
 \qquad(x\in X).
\]
This proves the first implication.

\emph{Step 2.}  Let $\varphi:F\to G$ be the presentation map.  Before
passing to $S/S'$, one has
\[
 i_1(ab)i_0(ab)^{-1}
 =\bigl(i_1(a)i_0(a)^{-1}\bigr)
  i_0(a)\bigl(i_1(b)i_0(b)^{-1}\bigr)i_0(a)^{-1}.
\]
Hence
\[
 \Delta(ab)=\Delta(a)+\overline a\,\Delta(b),
\]
so $\Delta$ is a $\varphi$-derivation.  Applying
Lemma~\ref{lem:fox-coefficients} with $A=M$ and $d=\Delta$ gives
\begin{equation}\label{eq:fox-difference}
 \Delta(u)
 =\sum_{x\in X}
 \overline{\frac{\partial u}{\partial x}}\,\Delta(x).
\end{equation}
By Step~1, every coefficient in
\eqref{eq:fox-difference}, with $u=v$, belongs to
$\gideal^{n-1}$.  Therefore
\[
 \Delta(v)\in\gideal^{n-1}M.
\]
This proves the second implication.

\emph{Step 3.}  Multiplication by the augmentation ideal raises
commutator weight in the relation module.  Indeed, let
$s\in S\cap\gamma_{m+1}(H)$, let $g\in G$, and choose a lift $h\in H$
of $g$.  Then
\[
 (g-1)(sS')=hsh^{-1}s^{-1}S'=[h^{-1},s^{-1}]S'.
\]
The last representative lies in $S\cap\gamma_{m+2}(H)$.  Since the
elements $g-1$ generate $\gideal$ additively, this proves
\[
 \gideal\,\frac{(S\cap\gamma_{m+1}(H))S'}{S'}
 \subseteq
 \frac{(S\cap\gamma_{m+2}(H))S'}{S'}.
\]
Starting with $M=(S\cap\gamma_1(H))S'/S'$ and applying this inclusion
inductively gives
\[
 \gideal^mM\subseteq
 \frac{(S\cap\gamma_{m+1}(H))S'}{S'}
 \qquad(m\geq0).
\]
Taking $m=n-1$ and using Step~2 shows that $\Delta(v)$ is represented
modulo $S'$ by an element of $S\cap\gamma_n(H)$.  Hence
\[
 i_1(v)i_0(v)^{-1}\in S'\gamma_n(H).
\]
This proves the third implication.

Both $S'$ and $\gamma_n(H)$ are normal in $H$, so this is the required
equality of the two images of $v$ in $Q_n(H,S)$.  Its class therefore
belongs to the equalizer and hence to $\Lim\ Q_n$.
\end{proof}

\begin{proof}[Proof of the main Theorem]
Lemmas~\ref{lem:representatives} and~\ref{lem:stability} give directly
\[
 \Lim\ Q_n=
 \frac{F\cap(1+\rideal\fideal+\fideal^n)}{R'\gamma_n(F)}
 \quad\text{inside }Q_n(F,R).
\]
By Lemma~\ref{lem:formal}, the colimit is $G/\gamma_n(G)$, so the
boundary limit is the image of
$F\cap(1+\rideal\fideal+\fideal^n)$ in this quotient.
Lemma~\ref{lem:correction} says that its image in $G$ is $D_n(G)$.
Hence
\[
 \Blim\ Q_n=D_n(G)/\gamma_n(G).
\]
This is an equality of canonical subgroups of $G/\gamma_n(G)$.
If $\varphi:G\to H$ is a homomorphism, then the induced group-ring map
sends $\gideal_G^n$ into $\gideal_H^n$, while
$\varphi(\gamma_n(G))\subseteq\gamma_n(H)$.  Thus the induced map
$G/\gamma_n(G)\to H/\gamma_n(H)$ carries the displayed canonical
subgroup into the corresponding subgroup for $H$.  Since the colimit
isomorphism in Lemma~\ref{lem:formal} is induced by the presentation
maps to these quotients, the resulting square commutes.  Hence the
isomorphism is natural in $G$.
\end{proof}

\section{The Lie-ring case}

Let $L$ be a Lie ring over $\Z$, and let
$\Pres_{\mathrm{Lie}}(L)$ be its category of free presentations
\[
 0\longrightarrow R\longrightarrow F\longrightarrow L
 \longrightarrow0.
\]
The lower central series of $L$ is
\[
 \gamma_1(L)=L,\qquad
 \gamma_{k+1}(L)=[\gamma_k(L),L]\quad(k\geq1),
\]
and $[R,R]$ is the derived Lie ideal of $R$.  Let $U(L)$ be the
universal enveloping ring of $L$ and
$\omega(L)$ its augmentation ideal.  By the Poincar\'e--Birkhoff--Witt
Theorem over a Dedekind domain
\cite[Theorem~3.2(ii)]{BartholdiPassi2015}, the canonical map
$L\to U(L)$ is injective; we use it to regard $L$ as a Lie subring of
$U(L)$.
The dimension series of $L$ is defined by
\[
 \delta_n(L)=L\cap\omega(L)^n,
\]
(see \cite{BartholdiPassi2015}.)

\begin{theorem}\label{thm:lie}
For every Lie ring $L$ over $\Z$ and every $n\geq1$,
\[
 \Blim_{\Pres_{\mathrm{Lie}}(L)}\ 
 \frac{F}{[R,R]+\gamma_n(F)}
 \cong\frac{\delta_n(L)}{\gamma_n(L)}
\]
naturally in $L$.
\end{theorem}

\begin{proof}
The case $n=1$ is immediate, so let $n\geq2$.  The formal parts of the
group proof, namely the equalizer description of the limit and the
calculation of the colimit, apply verbatim in the category of Lie rings,
with the identity element replaced by $0$.
We explain the places where the algebraic argument becomes simpler.

Put
\[
 I=\omega(F),\qquad
 J=\ker\bigl(U(F)\longrightarrow U(L)\bigr).
\]
First, the universal property of the enveloping ring gives
\[
 J=U(F)RU(F),
\]
the two-sided ideal generated by $R$
\cite[Proposition~4.2]{BartholdiPassi2015}.  Since
$U(F)=\Z\oplus I$ as additive groups, for $u,v\in U(F)$ and $r\in R$
one has
\[
 urv=\varepsilon(u)rv+\bigl(u-\varepsilon(u)\bigr)rv,
\]
where the first summand belongs to $RU(F)$ and the second to $IJ$.
Hence
\[
 J\subseteq RU(F)+IJ.
\]
Choose a free generating set $X$ of $F$.  For $r\in R$ and $x\in X$,
\[
 rx=xr+[r,x]\equiv[r,x]\pmod{IJ},
\]
because $xr\in IJ$ and $[r,x]\in R$.  Induction on the length of a
word gives $rw\in R+IJ$ for every $r\in R$: if $w=xw'$, then
\[
 rw=x(rw')+[r,x]w',
\]
the first summand lies in $IJ$ because $rw'\in J$, and the second is
congruent modulo $IJ$ to an element of $R$ by the induction hypothesis
applied to $[r,x]\in R$ and the shorter word $w'$.  Since
$U(F)=\Z\langle X\rangle$, linearity gives
$RU(F)\subseteq R+IJ$.  Conversely, $R\subseteq J$ and $IJ\subseteq J$
because $J$ is a two-sided ideal.  Therefore
\begin{equation}\label{eq:lie-correction}
 J=R+IJ
\end{equation}
as additive groups.  This is the first simplification: in the group
case a lift has to be corrected multiplicatively by an element of the
relation subgroup; here one simply subtracts an element of $R$.

The inclusion
\[
 [R,R]+\gamma_n(F)\subseteq F\cap(IJ+I^n)
\]
is immediate: $[r,s]=rs-sr\in IJ$ for $r,s\in R$, and an iterated
bracket of length $n$ lies in $I^n$.  If
$\pi\colon U(F)\to U(L)$ is the quotient map, then
$\pi(I)=\omega(L)$ and therefore $\pi(I^n)=\omega(L)^n$.  For a
surjective homomorphism, the inverse image of the image of an ideal is
the sum of that ideal with the kernel; hence
\[
 \pi^{-1}\bigl(\omega(L)^n\bigr)
   =J+I^n=R+IJ+I^n.
\]
Thus an element of $\delta_n(L)$ has a lift $f\in F$ of the form
$f=r+a$, where $r\in R$ and $a\in IJ+I^n$.  The element
$f-r\in F\cap(IJ+I^n)$ is another lift.  Consequently,
\begin{equation}\label{eq:lie-image}
 \im\bigl(F\cap(IJ+I^n)\longrightarrow L\bigr)=\delta_n(L).
\end{equation}

Second, ordinary first-letter coefficients replace Fox calculus.
In the free associative ring $U(F)=\Z\langle X\rangle$, every
$a\in U(F)$ has the unique
expansion
\begin{equation}\label{eq:first-letter}
 a=\varepsilon(a)+\sum_{x\in X}x\,\partial_x(a).
\end{equation}
Only finitely many summands are nonzero.  If $b\in I$ and
$c\in U(F)$, uniqueness of the expansion gives
$\partial_x(bc)=\partial_x(b)c$.  Taking $c\in J$, and then taking
$c\in I^{n-1}$ and using linearity, proves
\begin{equation}\label{eq:lie-partials}
 \partial_x(IJ)\subseteq J,\qquad
 \partial_x(I^n)\subseteq I^{n-1}.
\end{equation}

We use this to identify the possible representatives of a limit
element.  Let $H=F_0*F_1\to L$ be the coproduct presentation, with
inclusions $i_0,i_1$, and write
\[
 I_H=\omega(H),\qquad
 J_H=\ker\bigl(U(H)\to U(L)\bigr).
\]
Suppose that $a\in I$ represents an element of the equalizer of the two
maps on $I/(IJ+I^n)$.  Then
\[
 i_0(a)-i_1(a)\in I_HJ_H+I_H^n.
\]
Here $U(H)$ is freely generated by the two copies of $X$; hence the
$i_0(x)$-coefficient of $i_1(a)$ is zero and that of $i_0(a)$ is
$i_0(\partial_x(a))$.  Taking this coefficient and using
\eqref{eq:lie-partials} gives
\[
 i_0\bigl(\partial_x(a)\bigr)\in J_H+I_H^{n-1}.
\]
After mapping to $U(L)$ and pulling back to $U(F)$, we obtain
\[
 \partial_x(a)\in J+I^{n-1}\qquad(x\in X).
\]
Since $a\in I$, formula \eqref{eq:first-letter} now yields
$a\in IJ+I^n$.  Hence
\begin{equation}\label{eq:lie-vanishing}
 \Lim\ \frac{I}{IJ+I^n}=0.
\end{equation}
The inclusion
$[R,R]+\gamma_n(F)\subseteq F\cap(IJ+I^n)$ proved above defines the
quotient map used in the next argument.
The natural injection
\[
 \frac{F}{F\cap(IJ+I^n)}\longrightarrow\frac{I}{IJ+I^n}
\]
induces an injection on limits of the underlying pointed-set-valued
functors.  Its target limit is zero by \eqref{eq:lie-vanishing}, so
\[
 \Lim\ \frac{F}{F\cap(IJ+I^n)}=0.
\]
The quotient map from $F/([R,R]+\gamma_n(F))$ to the last functor now
shows that every element of
\(\Lim\ F/([R,R]+\gamma_n(F))\) is represented by an element of
$F\cap(IJ+I^n)$.

It remains to show that all these representatives are stable.  Put
$S=\ker(H\to L)$ and $M=S/[S,S]$.  The rule
\[
 (s+[S,S])\cdot\ell=[s,\widetilde\ell]+[S,S]
\]
makes $M$ a right $U(L)$-module.  Indeed, changing
$\widetilde\ell$ by an element of $S$ changes the bracket by an element
of $[S,S]$, and the Jacobi identity shows both that $[S,S]$ is stable
under bracketing with $H$ and that this right Lie action extends to
$U(L)$.  If a bar denotes the image of
$U(F)$ in $U(L)$, expansion on the free generators gives, in $M$,
\begin{equation}\label{eq:lie-difference}
 i_1(v)-i_0(v)
 =\sum_{x\in X}\bigl(i_1(x)-i_0(x)\bigr)
   \cdot\overline{\partial_x(v)}.
\end{equation}
Indeed, both sides agree on $X$ and satisfy
\[
 d([a,b])=d(a)\cdot\bar b-d(b)\cdot\bar a,
\]
because
$\partial_x([a,b])=\partial_x(a)b-\partial_x(b)a$.  Since $F$ is the
free Lie ring on $X$, induction on Lie monomials shows that two such
derivations agreeing on $X$ agree on all of $F$; this proves
\eqref{eq:lie-difference}.

For $m\geq0$, let
\[
 M_m=\frac{(S\cap\gamma_{m+1}(H))+[S,S]}{[S,S]}\subseteq M.
\]
If $s\in S\cap\gamma_{m+1}(H)$ and $\ell\in L$, then
$[s,\widetilde\ell]\in S\cap\gamma_{m+2}(H)$.  Each $M_m$ is stable
under the right $U(L)$-action, and $\omega(L)$ is the two-sided ideal
generated by $L$; consequently $M_m\omega(L)\subseteq M_{m+1}$.
Since $M_0=M$, induction gives
\begin{equation}\label{eq:lie-module-filtration}
 M\omega(L)^m\subseteq
 \frac{(S\cap\gamma_{m+1}(H))+[S,S]}{[S,S]}
 \qquad(m\geq0).
\end{equation}
For $v\in F\cap(IJ+I^n)$, equation \eqref{eq:lie-partials} gives
$\partial_x(v)\in J+I^{n-1}$, and hence
$\overline{\partial_x(v)}\in\omega(L)^{n-1}$.  Equation
\eqref{eq:lie-difference} and \eqref{eq:lie-module-filtration} now give
\[
 i_1(v)-i_0(v)\in[S,S]+\gamma_n(H)
\]
by \eqref{eq:lie-module-filtration}.  Thus every such $v$ is stable.
Together with \eqref{eq:lie-vanishing}, this proves that the limit,
evaluated at $(F,R)$, is
\[
 \frac{F\cap(IJ+I^n)}{[R,R]+\gamma_n(F)}.
\]
The colimit is $L/\gamma_n(L)$ by the Lie-ring version of
Lemma~\ref{lem:formal}.  In view of \eqref{eq:lie-image}, the image of
the limit in the colimit is precisely
$\delta_n(L)/\gamma_n(L)$.

Finally, a homomorphism of Lie rings sends $\omega(L)^n$ into the
corresponding $n$th power of the augmentation ideal and sends
$\gamma_n(L)$ into the corresponding lower-central term.  All maps in
the construction above are induced by the canonical maps from free
presentations to $L/\gamma_n(L)$, so the displayed identification
commutes with homomorphisms of Lie rings.  It is therefore natural in
$L$.

\end{proof}

\bigskip
\noindent
Saint Petersburg State University\\
7/9 Universitetskaya nab., St.~Petersburg, 199034 Russia

\end{document}